\documentclass[11pt,reqno]{amsart}
\usepackage{amsmath, amssymb, graphicx, mathrsfs, hyperref}
\usepackage{verbatim} 
\usepackage[usenames,dvipsnames,svgnames]{xcolor}
\usepackage[lmargin=1in,rmargin=1in,tmargin=1in,bmargin=1in]{geometry}

\renewcommand{\Re}{\operatorname{Re}}

\renewcommand{\min}{\operatorname{min}}
\renewcommand{\max}{\operatorname{max}}

\renewcommand{\sup}{\operatorname{Sup}}
\newcommand{\s}{{\sigma}}

 \renewcommand{\a}{\alpha}
\renewcommand{\b}{\beta}

\renewcommand{\d}{{\delta}}

\newcommand{\G}{\Gamma}
\renewcommand{\l}{\lambda}

\newcommand{\z}{\zeta}

\newcommand{\fourpartdef}[8]
{
	\left\{
		\begin{array}{ll}
			#1 & \mbox{if } #2 \\
			#3 & \mbox{if } #4\\
			#5 & \mbox{if } #6\\
			#7 & \mbox{if } #8
		\end{array}
	\right.
}
\newcommand{\twopartdef}[4]
{
	\left\{
		\begin{array}{ll}
			#1 & \mbox{if } #2 \\
			#3 & \mbox{if } #4
		\end{array}
	\right.
}
\renewcommand{\(}{\left\(}
\renewcommand{\)}{\right\)}
\renewcommand{\pmod}[1]{\,(\textup{mod}\,#1)}
\numberwithin{equation}{section}
 \theoremstyle{plain}
\newtheorem{theorem}{Theorem}[section]
\newtheorem{lemma}[theorem]{Lemma}

\newtheorem{proposition}[theorem]{Proposition}

\allowdisplaybreaks[4]

\begin{document}
\title[Moments of averages of generalized Ramanujan sums]{Moments of averages of generalized Ramanujan sums}
\author{Nicolas Robles}
\address{Institut f\"{u}r Mathematik, Universit\"{a}t Z\"{u}rich Winterthurerstrasse 190 \\ CH-8057 Z\"{u}rich, Switzerland\\
}
\email{nicolas.robles@math.uzh.ch} 
\author{Arindam Roy}
 \address{Department of Mathematics, University of Illinois, 1409 W. Green Street \\ Urbana, IL 61801, United States}
\email{roy22@illinois.edu}
\maketitle
\begin{abstract}
	Let $\beta$ be a positive integer. A generalization of the Ramanujan sum due to Cohen is given by 
	\begin{align}
		c_{q,\beta }(n) := \sum\limits_{{{(h,{q^\beta })}_\beta } = 1} {{e^{2\pi inh/{q^\beta }}}}, \nonumber
	\end{align}
	where $h$ ranges over the the non-negative integers less than $q^{\beta}$ such that $h$ and $q^{\beta}$ have no common $\beta$-th power divisors other than $1$. 
 The distribution of the average value of the Ramanujan sum is a subject of extensive research. In this paper, we study the distribution of the average value of $c_{q,\beta }(n)$ by computing the $k$-th moments of the average value of $c_{q,\beta }(n)$. In particular we have provided the first and second moments with improved error terms. We give more accurate results for the main terms than our predecessors. We also provide an asymptotic result for an extension of a divisor problem and for an extension of Ramanujan's formula.   
\end{abstract}
\section{Introduction}
In \cite{RamanujanTrigonometrical}, Ramanujan introduced a trigonometrical sum
\begin{align} 
{c_q}(n) :=\sum\limits_{(h,q) = 1}\cos\left(\frac{2\pi nh }{q}\right)= \sum\limits_{(h,q) = 1} {{e^{2\pi inh/q}}},\label{ramanujansum}
\end{align}
where $q$ and $n$ are positive integers. 
Ramanujan sums fit naturally with other arithmetical functions. For instance, one has 
\begin{align}
c_q(1)=\mu(q)\quad\mbox{and}\quad  c_q(q)=\phi(q), \nonumber
\end{align}
where $\mu(n)$ and $\phi(n)$ are the M\"{o}bius and Euler totient functions, respectively. Moreover, if $(q,r)=1$, then $c_q(n)c_r(n)=c_{qr}(n)$. In the same article, Ramanujan obtained expressions of the form 
\begin{align}
f(n)=\sum_{q=1}^{\infty}a_qc_q(n)\label{ramexpres}
\end{align}
for some arithmetical functions $ a_q $. In particular, 
\begin{align} 
  \sum_{q=1}^{\infty} {\frac{{{c_q}(n)}}{q}}=0, \quad \sum_{q=1}^{\infty} {\frac{{\log q}}{q}{c_q}(n)}=-d(n),\quad\sum_{q=1}^{\infty} {\frac{{{{( - 1)}^{q - 1}}{c_{2q - 1}}(n)}}{{2q - 1}}}=\frac{1}{\pi}r(n) \label{formulae1}
\end{align}
as well as
\begin{align}
\sum_{q=1}^{\infty} {\frac{{{c_q}(n)}}{{{q^{s+1}}}}}=\frac{\sigma_{-s} (n)}{\z(s+1)}\quad\text{and}\quad
\sum_{q=1}^{\infty} {\frac{{\mu (q){c_q}(n)}}{{{\phi _{s+1}}(q)}}}=\z(s+1)\frac{\phi_s(n)}{n^s}\label{formulae2}
\end{align}
for $\Re(s)>0$. 
Here $d(n)$ is the number of divisor of $n$, $\sigma_s(n)$ the sum of their $s$-th powers,
\begin{align}
{\phi _s}(n) = {n^s}\left( {1 - \frac{1}{{{p_1^s}}}} \right)\left( {1 - \frac{1}{{p_2^s}}} \right) \cdots\left( {1 - \frac{1}{{p_k^s}}} \right) \nonumber
 \end{align}
when $n=p_1^{a_1} p_2^{a_2}\cdots p_k^{a_k}$, 
and $r(n)$ is the number of representations of $n$ as the sum of two squares. Also he showed that
\begin{align}
  \sum\limits_{n = 1}^\infty  {\frac{{{c_q}(n)}}{n}}=-\Lambda (q)\quad\text{and}\quad\sum_{\substack{d\mid n \\ d\mid q}}d\mu(q/d)=c_q(n),\label{formulae3}
\end{align} 
where $\Lambda(n)$ is the von Mangoldt function. 

The second equation of \eqref{formulae1} is of the same depth as the prime number theorem. As discussed by Hardy and Wright in \cite{HardyWright}, these series have a particular interest because they show explicitly the source of the irregularities in the behavior of their sums.
Note that the Ramanujan expansion \eqref{ramexpres} mimics the notion of a Fourier expansion of an $L^1$-function. 
In \cite{Carmichael}, Carmichael  noticed an orthogonality principle of Ramanujan sums. This allows one to predict the Ramanujan coefficients $a_q$ in \eqref{ramexpres} of an arithmetical function $f(n)$ if such expansion exists. The work of 
Wintner \cite{Wintner}  and Delange \cite{DelangeRam} allows us to determine a large number of Ramanujan expansions.
Later on more work was done in this direction by  Delange \cite{Delange}, Wirsing \cite{Wirsing}, Hildebrand \cite{Hildebrand}, Schwarz \cite{Schwarz}, Lucht and Reifenrath \cite{LuchtReifenrath}.

Ramanujan sums and their variations make surprising appearances in singular series of the Hardy-Littlewood asymptotic formula for Waring problems and in the asymptotic formula of Vinogradov on sums of three primes, for details the reader is referred to \cite{Davenport}.

Recently, Alkan \cite{Alkan} studied the weighted averages of Ramanujan sums. He showed that for integer $r \ge 1$ and $x \ge 1$ one has
\begin{align} 
& \sum_{k \le x} \bigg( \frac{1}{k^{r+1}} \sum_{j=1}^k j^r c_k(j) \bigg) = 1 + \frac{1}{2} \sum_{2 \le k \le x} \frac{\phi(k)}{k} + \frac{1}{1+r} \sum_{m=1}^{[\tfrac{r}{2}]} \binom{r+1}{2m} B_{2m} \sum_{2 \le k \le x} \prod_{p|k} \bigg(1 -\frac{1}{p^{2m}} \bigg),  \nonumber 
\end{align}
where $B_{2m} \ne 0$ are the Bernoulli numbers together with the convention that the sum over $m$ is taken to be zero when $r = 1$ and the sums over $k$ are taken to be zero when $1 \le x < 2$.

In \cite{ChanKumchev}, Chan and Kumchev studied moments of averages of Ramanujan sums. They showed that for $y \ge x$ one has
\begin{align} \label{chankumchev1}
\sum_{n \le y} \sum_{q \le x} c_q(n) = y - \frac{x^2}{4\zeta(2)} + O(xy^{1/3}\log x +x^3y^{-1}),
\end{align}
as well as
\begin{align} 
\sum_{n \le y} \bigg(\sum_{q \le x} c_q(n)\bigg)^2 = \frac{yx^2}{2\zeta(2)} + O(x^4 + xy \log x) \nonumber 
\end{align}
for $y \ge x^2 (\log x)^B$ for $B>0$, and lastly for $x \le y \le x^2(\log x)^B$
\begin{align} \label{chankumchev2b}
\sum_{n \le y} \bigg(\sum_{q \le x} c_q(n)\bigg)^2 = \frac{yx^2}{2\zeta(2)}(1+2 \kappa(u)) + O(yx^2 (\log x)^{10} (x^{-1/2}+(y/x)^{-1/2}) ),
\end{align}
where $u = \log(yx^{-2})$ and $\kappa(u)$ is a certain Fourier integral given by 
\begin{align} 
\kappa(u) := \frac{1}{2\pi} \int_{-\infty}^{\infty} f(it)e^{-itu}dt, \quad \textnormal{where} \quad f(s) := \frac{\zeta(1-s)}{\zeta(1+s)}\frac{1}{(1+s)^2(1-s)}. \nonumber 
\end{align}
It satisfies some numerical inequalities given in \cite{ChanKumchev} and in particular $\kappa(u)=o(1)$. 

Let $\beta$ be a positive integer. A generalization of the Ramanujan sum due to Cohen \cite{CohenExtension} is written as $c_{q,\beta }(n)$ and it is defined by
\begin{align} 
c_{q,\beta }(n): = \sum\limits_{{{(h,{q^\beta })}_\beta } = 1} {{e^{2\pi inh/{q^\beta }}}},\label{cohendef}
\end{align}
where $h$ ranges over the the non-negative integers less than $q^{\beta}$ such that $h$ and $q^{\beta}$ have no common $\beta$-th power divisors other than $1$.
It follows immediately that when $\beta = 1$, \eqref{cohendef} becomes the Ramanujan sum \eqref{ramanujansum}. Clearly this generalization of the Ramanujan sum is as important as Ramanujan sum by its arithmetic nature. For more arithmetic properties of the generalized Ramanujan sum \eqref{cohendef}, the reader is referred to \cite{CohenExtension}. For a discussion of the connections between the generalized Ramanujan sums due to Cohen and the non-trivial zeros of the Riemann zeta-function the reader is referred to \cite{KuhnRobles}.

Let us now introduce the main object of study of this paper. The $k^{\operatorname{th}}$ moment of the average of the generalized Ramanujan sum \eqref{cohendef} is defined by
\begin{align}
 C_{k,\beta}(x,y):= \sum\limits_{n \leqslant y} {{{\bigg( {\sum\limits_{q \leqslant x} {{c_{q,\beta}}(n)} } \bigg)}^k}},\label{kthmoment}
\end{align}
where $k$ is a positive integer and $x$ and $y$ are reals.
It is not to difficult to obtain an asymptotic result for the $k^{\operatorname{th}}$ moment of \eqref{kthmoment}. In particular we have
\begin{proposition}\label{rr01}
Let $k$ and $\beta$ be two positive integers. Let $y>x^{k(1+\beta)}\log^{k+1}x$, then 
\begin{align}
C_{k,\beta }(x,y) &= A_{k,\beta}(x,y) + O\left(x^{k(1+\beta)}\log^k x\right), \nonumber 
\end{align}
where
\begin{align}
A_{k,\beta}(x,y)=\twopartdef{y,}{k=1,}{\frac{{y{x^{1 + \beta }}}}{{(1 + \beta )\zeta (1 + \beta )}}+O(yx^\b\log^{\lfloor 1/\b\rfloor}x),}{k>1} \nonumber
\end{align}
\end{proposition}
For the first and second moments one can improve the error terms as well as clarify the dependence between the parameters $y$ and $x$. Our main results are following.
\begin{theorem} \label{maintheorem}
Let $y \ge  x^{3\beta/2}\log^5 x$. Then for $\beta=1,2$  one has
\begin{align} 
C_{1,\beta }(x,y) = y - \frac{{{x^{1 + \beta }}}}{{2(1 + \beta )\zeta (1 + \beta )}} + O(x^\b y^{1/3}\log^4y+x^{2\b+1}y^{-2/3}+x^{\b+1}y^{-1/3}). \nonumber 
\end{align}
and for $\b\ge 3$ one has
\begin{align} 
C_{1,\beta }(x,y) = y + O(x^\b y^{1/3}\log^4y).\nonumber 
\end{align}
\end{theorem}
\begin{theorem}\label{maintheorem02}
For $\b=1,2$ and $x^{2\b}<y<x^{2\b+\b^2}\log^{\frac{5}{2}(\b+1)}x$ one has
\begin{align} 
C_{2,\beta }(x,y) &= \frac{{y{x^{1 + \beta }}}}{{(1 + \beta )\zeta (1 + \beta )}} - \frac{1}{2}\frac{{{x^{2 + 2\beta }}}}{{{{(1 + \beta )}^2}{\zeta ^2}(1 + \beta )}}+O(y^{-1}x^{2+4\b}+x^{2\b+1}(\log x +\log\log x)) \nonumber \\&\quad+O\bigg(x^{2\b}y^{\frac{1}{3}+\frac{1}{6\b}}(\log^5y)\log\log y(\log^4 x+\log^4\log x)+yx^{\frac{1}{2}+\b}(\log^3x+\log^3\log x)\bigg). \nonumber 
\end{align}
For $\b\ge 3$ and $y>x^{3\b/2}$ one has
\begin{align} 
C_{2,\beta }(x,y) = \frac{{y{x^{1 + \beta }}}}{{(1 + \beta )\zeta (1 + \beta )}}&+O\bigg(x^{2\b}y^{\frac{1}{3}+\frac{1}{6\b}}(\log^5y)\log\log y(\log^4 x+\log^4\log x)\bigg) \nonumber \\
&+O\bigg(yx^{\frac{1}{2}+\b}(\log^3x+\log^3\log x)\bigg). \nonumber
\end{align}
\end{theorem}

{\bf Remarks:} (i) Theorem \ref{maintheorem02} is not only more general but also improves \eqref{chankumchev2b} when we choose $\b=1$.

(ii) In order to improve Proposition \ref{rr01} for $k\geq 3$ one may need to assume some strong results such as the moment hypothesis of the Riemann zeta-function. For  $k\geq 3$ improving Proposition \ref{rr01} unconditionally is still an open question.

Next we consider a generalization of the divisor function, defined by 
\begin{align}
\sigma_{z,\beta}(n) := \sum\limits_{{d^\beta }|n} {{d^{\beta z}}}.\label{sigbz} 
\end{align}
Crum \cite{CrumDirichletseries} seems to be first author who coined the notation \eqref{sigbz}. Understanding the asymptotic behavior of sums like
\begin{align}
\sum_{n\leq x}\s_{z_1,\b}(n)\quad\text{and}\quad\sum_{n\leq x}\s_{z_1,\b}(n)\s_{z_2,\b}(n) \nonumber
\end{align}
for $\Re(z_i)\leq 0$, $i=1,2$ is naturally needed in the proofs of Theorems \ref{maintheorem} and \ref{maintheorem02}. However, these sums are important objects in their own right. Clearly these sums are generalizations of 
\begin{align}
\sum_{n\leq x}d(n)\quad\text{and}\quad\sum_{n\leq x}d^2(n), \nonumber
\end{align}
respectively. 

%
The evaluation of the summation of the divisor function
\begin{align}
D(x) := \sum_{n \le x}d(n), \nonumber
\end{align}
has been studied extensively in the literature. In particular, it can be shown that
\begin{align}
D(x) = x \log x + (2 \gamma -1)x + \Delta(x), \nonumber
\end{align}
and the specific determination of the error term $\Delta(x)$ is called the Dirichlet divisor problem (see \cite[p. 68]{MontgomeryVaughanBook}). In 1849, Dirichlet \cite{Dirichler1849} proved that $\Delta(x)$ could be taken to be $O(x^{1/2})$. Further progress came in 1903 by Vorono\"{i} \cite{Voronoiasym}, who showed that $\Delta(x) \ll x^{1/3}\log x$, and then by van der Corput who proved in 1922 that $\Delta(x) \ll x^{33/100+\varepsilon}$, \cite{vanderCorput}. The exponent has been reduced over the years (see \cite[p. 69]{MontgomeryVaughanBook} for further details). The current record stands at $\Delta(x) \ll x^{131/416 + \varepsilon}$ and it is due to Huxley  \cite{huxley2}.

On the other hand, in \cite{Ramanujan16}, Ramanujan states without proof that
\begin{align} \label{ramanujanunproved}
d^2(1) + d^2(2) + \cdots + d^2(n) = A n (\log n)^3 + Bn (\log n)^2 + C n \log n + Dn +O(n^{3/5+\varepsilon}).
\end{align}
Moreover, and also without proof, Ramanujan claims that on the Riemann hypothesis, the error term in \eqref{ramanujanunproved} can be strengthened to $O(n^{1/2+\varepsilon})$. In 1922, Wilson \cite{Wilson} proved that indeed one can take the error term to be $O(n^{1/2+\varepsilon})$ unconditionally. As can be seen from \cite{MontgomeryVaughan1981, Ramachandra1999}, it is highly probable that that the error term is $O(n^{1/2})$.
Suppose that $P_3(t)$ denotes a polynomial in $t$ of degree 3. Let us set the notation 
\begin{align}
E(x) := \sum_{n \le x}d^2(n) - x P_3(\log x). \nonumber
\end{align}
Ramachandra and Sankaranarayanan \cite{RamachandraSankaranarayanan} showed that
\begin{align}
E(x) = O(x^{1/2} (\log x)^5 (\log \log x)) \nonumber
\end{align}
unconditionally.
In 1962, Chandrasekharan and Narasimhan \cite{ChandrasekharanNarasimhan} proved that 
\begin{align}
E(x) = \Omega_{\pm}(x^{1/4}). \nonumber 
\end{align}


For a given arithmetic function $f(n)$ we define 
\begin{align}
\sideset{}{'}\sum_{n\leq x}f(n)=\sum\limits_{n \leqslant x} f(n) - \frac{1}{2}f(x), \nonumber  
\end{align}
when $x$ a is positive integer. We have following asymptotic results.
\begin{theorem} \label{theormsigma}
Let $\operatorname{Re}(z)\leq 0$. Then
\begin{align}
\sideset{}{'}\sum\limits_{n \leqslant x} {\sigma_{z,\beta}(n)} = D_{z,\b}(x)+\Delta_{z,\beta}(x), \nonumber
\end{align}
where 
\begin{align}
\Delta_{z,\beta}(x)\ll x^{\frac{1}{3}}\log^2 x \nonumber 
\end{align}
uniformly for $\beta\geq 1$ and $ D_{z,\b}(x) $ is given by following.

(i) If $\b=1,2$ and $-\tfrac{2}{3\b^2}<\Re(z)\leq 0$, then 
\begin{align}
D_{z,\b}(x)=\zeta (\beta (1 - z))x + \frac{1}{{1 + \beta z}}\zeta \left( {z + \frac{1}{\beta }} \right){x^{z + \tfrac{1}{\beta }}} . \nonumber
\end{align}

(ii) If $\b\geq 3$ and $-1<\Re(z)\leq 0$, then 
\begin{align}
D_{z,\b}(x)=\zeta (\beta (1 - z))x . \nonumber
\end{align}
\end{theorem}
\begin{theorem} \label{theormsigma2}
Let $ \Re(z_1), \Re(z_2)\leq 0$, $\Re(z_1+z_2)>-1$ and $|\Re(z_1-z_2)|<1/b$. Then for $\b\geq 1$ one has
\begin{align}
\sideset{}{'}\sum\limits_{n \leqslant x} {\sigma _{z_1,\beta }(n)\sigma _{z_2,\beta}(n)} &= D_{z_1,z_2, \b}(x)+\Delta_{z_1,z_2,\b}(x), \nonumber
\end{align}
where the values $  D_{z_1,z_2, \b}(x)$ and $\Delta_{z_1,z_2,\b}(x) $ are given below.

(i) If  $\b=1,2$ and $-\tfrac{1}{2(2\b+1)}<\Re(z_1),\Re(z_2), \Re(z_1+z_2)\leq 0$ then for $z_1\neq 0, z_2\neq 0$, and $z_1\neq z_2$ one has
\begin{align}
 D_{z_1,z_2, \b}(x)=&\frac{{\zeta (\b(1 - z_1))\zeta (\b(1 - z_2)\zeta (\b(1 - z_1 - z_2))}}{{\zeta  (\b(2 - z_1 - z_2))}}x \nonumber \\
&\qquad+\frac{{\zeta (z_1 + 1/\b)\zeta (1 + \b z_1 -  \b z_2)\zeta (1 - \b z_2)}}{{(\b z_1 + 1)\zeta (2 + \b z_1 - \b z_2)}}x^{z_1+\frac{1}{\b}} \nonumber \\
&\qquad+\frac{{\zeta (z_2 + 1/\b)\zeta (1 +\b z_2 -\b z_1)\zeta (1 -\b z_1)}}{{(\b z_2 + 1)\zeta (2 +\b z_2  -\b  z_1)}}x^{z_2+\frac{1}{\b}} \nonumber \\
&\qquad+\frac{{\zeta (z_1 + z_2 + 1/\b)\zeta (\b z_2 + 1)\zeta (\b z_1 + 1)}}{{(\b z_1 +\b  z_2 + 1)\zeta (2 +\b  z_1 +\b z_2)}}x^{z_1+z_2+\frac{1}{\b}}, \nonumber
\end{align}
and
\begin{align} 
\Delta_{z_1,z_2,\b}(x)\ll x^{\frac{1}{3}+\frac{1}{6\b}+\frac{\operatorname{Re}(z_1)+\operatorname{Re}(z_2)}{6}}(\log^5x)\log\log x. \nonumber
\end{align}

(ii) If $\b\geq 3$ then for $z_1\neq 0, z_2\neq 0$, and $z_1\neq z_2$, then
\begin{align}
D(z_1,z_2,\b)(x)= \frac{{\zeta (\beta (1 - z_1))\zeta (\beta (1 - z_2))\zeta (\beta (1 - z_1 - z_2))}}{{\zeta (\beta (2 - z_1 - z_2))}}x, \nonumber
\end{align}
and 
\begin{align}
\Delta_{z_1,z_2,\b}(x)\ll \max\bigg(x^{\tfrac{1}{2\b}+\tfrac{\operatorname{Re}(z_1)+\operatorname{Re}(z_2)}{2}+\tfrac{\b|\operatorname{Re}(z_1)-\operatorname{Re}(z_2)|}{3}},x^{\tfrac{1}{3}+\tfrac{1}{6\b}+\tfrac{\operatorname{Re}(z_1)+\operatorname{Re}(z_2)}{6}}\bigg)(\log^5x)\log\log x. \nonumber
\end{align}
\end{theorem}
{\bf Remark:} For other values of $z_1$ and $z_2$, such as $z_i=0$ or $z_1=z_2$, one can compute explicitly the value of $D_{z_1, z_2, \b}(x)$. 
Since the other cases are not of interest in the present paper,
only values of $D_{z_1, z_2, \b}(x)$ for $\Re(z_i)\leq 0$, $i=1,2$ are provided. In the proof of Theorem \ref{theormsigma2} we will see how the other values can, in fact, be obtained. Theorems \ref{theormsigma} and \ref{theormsigma2} can be computed for $\Re(z_i)>0$, $i=1,2$ by similar arguments of the methods presented here. We also avoid these cases. 

%
\section{Preliminaries}
We start this section by recalling two important identities due to Cohen \cite{CohenExtension}. These identities generalize the first identity of \eqref{formulae2} and the second identity of \eqref{formulae3}.
\begin{lemma}\label{cohen01}
	Suppose that $\beta$ is a positive integer, then one has
	\begin{align}
	\sum\limits_{q = 1}^\infty  {\frac{{c_{q,\beta}(n)}}{{{q^{\beta s}}}}}  = \frac{{\sigma _{1 - s,\beta }(n)}}{{\zeta (\beta s)}} \nonumber
	\end{align}
	for $\Re(s)>1$.
\end{lemma}

\begin{lemma}\label{cohen02}
The generalization of the Ramanujan sum may be written as
\begin{align} 
c_{q,\beta}(n) = \sum_{\substack{d|q \\ {d^\beta }|n}} {{d^\beta }\mu \left( {\frac{q}{d}} \right)},\label{cohenmoebius}
\end{align}
where $\mu(n)$ is the M\"{o}bius function.
\end{lemma}
In \cite{CrumDirichletseries}, Crum derived the Dirichlet series for $\sigma _{z,\beta }(n)$.
\begin{lemma}\label{crum01}
	Suppose that $\beta$ is a positive integer and that $z \in \mathbb{C}$. One has
	\begin{align}
	\sum\limits_{n = 1}^\infty  {\frac{{\sigma _{z,\beta }(n)}}{{{n^s}}}}  = \zeta (s)\zeta (\beta (s - z)). \nonumber
	\end{align}
	for $\Re(s) > \max(\Re(z) +\frac{1}{\beta}, 1)$. 
\end{lemma}
The Dirichlet series for $\sigma_{p,\beta}(n)\sigma_{q,\beta}(n)$ is given in \cite{CrumDirichletseries}. 
\begin{lemma}\label{crum02}
	One has
	\begin{align}
	\sum\limits_{n = 1}^\infty  {\frac{{\sigma _{z_1,\beta}(n)\sigma _{z_2,\beta }(n)}}{{{n^s}}}}  = \frac{{\zeta (s)\zeta (\beta (s - z_1))\zeta (\beta (s - z_1))\zeta (\beta (s - z_1 - z_2))}}{{\zeta (\beta (2s - z_1 - z_2))}} \nonumber
	\end{align}
	for $\Re (s) > \max ( 1,\Re (z_1) + 1/\beta  ,\Re (z_2) + 1/\beta ,\Re (z_1 + z_2) + 1/\beta )$. 
\end{lemma}
From \cite[Lemma 4.5, page 72]{TitchmarshHeath-Brown} we have
\begin{lemma}\label{expint}
Let $a,b$ and $M$ be real numbers and $r>0$. Let $F$ be a real valued function, twice differentiable, and $|F''(x)|\geq r$ in $[a,b]$. Let $G$ be a real valued function, $G/F'$ be monotonic and $|G(x)|\leq M$. Then 
\begin{align}
\left\lvert\int_{a}^{b}G(x)e^{iF(x)}dx\right\rvert\leq\frac{8M}{\sqrt{r}}. \nonumber
\end{align}

\end{lemma}
Let $\mathcal{M}$ be a class of non-negative arithmetic functions which are multiplicative and that satisfy: 
\begin{itemize}
\item[(i)]
there exists a positive constant $A$ such that if $p$ is a prime and $l\geq 1$ then
\begin{align}
f(p^l)\leq A^l; \nonumber
\end{align}
\item[(ii)]
for every $\epsilon>0$, there exists a positive constant $B(\epsilon)$ such that 
\begin{align}
f(n)\leq B(\epsilon)n^{\epsilon} \nonumber
\end{align}
for $n\geq 1$.
\end{itemize} 
In \cite{Shiu}, Shiu showed
\begin{lemma}\label{shiu}
Let $f\in \mathcal{M}$, $0<\alpha,\beta<1/2$ and let $a, k$ be integers. If $0<a<k$ and $(a,k)=1$, then as $x\to\infty$
\begin{align}
\sum_{\substack{x-y<n\leq x\\n\equiv a\pmod q}}f(n)\ll \frac{y}{\phi(q)\log x}\exp\bigg(\sum_{p\leq x,\  p\nmid q}\frac{f(p)}{p}\bigg), \nonumber
\end{align}
uniformly in $a, q$, and $y$ provided that $q\leq y^{1-\alpha}$, $x^{\beta}<y\leq x$. 
\end{lemma}
In \cite{NairTenenbaum}, Nair and Tenenbaum observed that if $q=1$ then one can obtain the same result when $f$ is non-negative, sub-multiplicative and satisfying (i) and (ii). 

We also recall the following well-known estimate. 
\begin{lemma}\label{merest}
Let $M$ be the Mertens constant. Then
\begin{align}
\sum_{p\leq x}\frac{1}{p}=\log\log x +M+O\bigg(\frac{1}{\log x}\bigg). \nonumber
\end{align}
\end{lemma}
The following lemma can be easily adopted from \cite[Theorem 5.2]{MontgomeryVaughanBook}. For the sake of completeness we will give the sketch of the proof.
\begin{lemma}\label{montvon}
Let $0<\l_1<\l_2<\dots<\l_n\to\infty$ be any sequence of real numbers and let $ \{a_n\} $ be any sequence of complex numbers. Let the Dirichlet series $\alpha(s) := \sum_{n=1}^{\infty} a_n \l_n^{-s}$ be absolutely convergent for some $\Re(s)>\sigma_a$. If $\sigma_0 > \max(0,\sigma_a)$ and $x>0$, then
\begin{align}
\sideset{}{'}\sum_{\l_n \le x} a_n = \frac{1}{2\pi i}\int_{\sigma_0 -iT}^{\sigma_0 + iT} \alpha(s) \frac{x^s}{s}ds + R, \nonumber
\end{align}
where
\begin{align}
R \ll \sum_{\substack{x/2 < \l_n < 2x \\ n \ne x}} |a_n| \min \bigg( 1, \frac{x}{T|x-\l_n|} \bigg) + \frac{4^{\sigma_0}+x^{\sigma_0}}{T} \sum_{n=1}^{\infty} \frac{|a_n|}{\l_n^{\sigma_0}}. \nonumber
\end{align}
\end{lemma}
\begin{proof}
Let
\begin{align}
\operatorname{si}(x) :=-\int_{x}^{\infty}\frac{\sin u}{u}du. \nonumber 
\end{align}
Integrating by parts one obtains 
\begin{align}
\operatorname{si}(x)\ll \min(1,1/x)\label{siest}
\end{align}
for $x>0$.
The proof of the lemma follows from the following identity \cite[p.~139, Eq.~(5.9)]{MontgomeryVaughanBook}
\begin{align}
\frac{1}{2\pi i}\int_{\sigma_0 -iT}^{\sigma_0 + iT} \frac{x^s}{s}ds\ll \fourpartdef{1+O(\frac{x^{\s_0}}{T}),}{x\geq 2}{1+\frac{1}{\pi}\operatorname{si}(T\log x)+O(\frac{2^{\s_0}}{T}),}{1\leq x\leq 2}{-\frac{1}{\pi}\operatorname{si}(-T\log x)+O(\frac{2^{\s_0}}{T}),}{\frac{1}{2}\leq x\leq 1}{O(\frac{x^{\s_0}}{T}),}{x\leq \frac{1}{2}}\label{fourpartiden}
\end{align}
for $\s_0>0$. Now 
\begin{align}
\frac{1}{2\pi i}\int_{\sigma_0 -iT}^{\sigma_0 + iT} \alpha(s) \frac{x^s}{s}ds=\sum_{n=1}^{\infty}a_n\frac{1}{2\pi i}\int_{\sigma_0 -iT}^{\sigma_0 + iT} \frac{(x/\l_n)^s}{s}ds.\label{id1}
\end{align}
Applying \eqref{siest}, \eqref{fourpartiden} and the fact that
\begin{align}
|\log(1+\d)|\asymp|\d| \nonumber
\end{align}
in \eqref{id1} we obtain the desired result.
\end{proof}
Note that for any fixed real number $t'$ and $\s\geq 1/2$ (see \cite[Chap.~VII]{TitchmarshHeath-Brown})
\begin{align}
\int_{T/2}^{T}|\z^4(\s+i(t+t'))|dt \ll T\log^4 T. \nonumber
\end{align}
Therefore arguing in a similar fashion as in \cite[Lemmas 3.3 and 3.4]{RamachandraSankaranarayanan} we have following two lemmas.
\begin{lemma}\label{ramsan01}
Let $\sigma\geq 1/2$ and $T>0$. Then for any fixed real numbers $t', t''$, and $\s'$ we have 
\begin{align}
\int_{|t| \le T} \bigg| \frac{\zeta^4(\s+i(t+t'))}{\zeta(1+2it)} \bigg| \frac{dt}{|\s'+i(t+t'')|}  \ll (\log T)^5 (\log \log T). \nonumber
\end{align}
\end{lemma}
\begin{lemma}\label{ramsan02}
Let $z$ be a complex number and $\Re(z)>0$. For $\sigma \ge 1/2$ we have
\begin{align}
\int_{1/2}^1 \int_{T/2}^T \bigg| \frac{\zeta^4(s+z)}{\zeta(2s)} \bigg| \bigg| \frac{x^s}{s} \bigg| d\sigma dt \ll (\log T)^4 (\log \log T)(x-x^{1/2})(\log x)^{-1}. \nonumber
\end{align}
\end{lemma}
\section{Proof of Proposition \ref{rr01}}
From \eqref{cohenmoebius} we see that
\begin{align}
C_{k,\beta }(x,y) = \sum\limits_{n \leqslant y} {{{\bigg( {\sum\limits_{q \leqslant y} {c_q^{(\beta )}(n)} } \bigg)}^k}}  
&=\prod_{j=1}^{k}\sum\limits_{{d_j}{k_j} \leqslant x} d_j^{\beta}\mu(k_j) \sum_{\substack{n\leq y\\d_j^{\beta}\mid n}}1
&=\prod_{j=1}^{k}\sum\limits_{{d_j}{k_j} \leqslant x} d_j^{\beta}\mu(k_j)\left\lfloor\frac{y}{[d_1^{\beta},\dots,d_k^{\beta}]}\right\rfloor, \nonumber
\end{align}
where $[d_1^{\beta},\dots,d_k^{\beta}]$ denotes the least common multiple of the integers $d_1^{\beta}, d_2^{\beta},\dots, d_k^{\beta}$. Let $(d_1^{\beta}, \dots, d_k^{\beta})$ denote the greatest common divisor of the integers $d_1^{\beta}, d_2^{\beta},\dots, d_k^{\beta}$. Then one derives
\begin{align}
C_{k,\beta }(x,y) &= y\prod_{j=1}^{k}\sum\limits_{{d_j}{k_j} \leqslant x}(d_1^{\beta}, \dots, d_k^{\beta})\mu(k_j) + O\bigg(\prod_{j=1}^{k}\sum\limits_{{d_j}{k_j} \leqslant x} d_j^{\beta} \bigg)  \label{ckbeq} 
\end{align}
for $k\geq 2$. If $k=1$, then
\begin{align}
C_{1,\b}(x,y)=y\sum_{dk\leq x}\mu(k)+O\bigg(\sum_{dk\leq x}d^\b\bigg)=y+O\bigg(\sum_{dk\leq x}d^\b\bigg). \nonumber
\end{align}
By the aid of the fact that
\begin{align}\sum\limits_{n \leqslant x} {{n^\beta }}  = x^{1 + \beta } + O(x^{\beta}) \nonumber \end{align}
we deduce that
\begin{align}
\sum\limits_{dk \leqslant x} {{d^\beta }\mu (k)} = \sum\limits_{k \leqslant x} {\mu (k)\left( {\frac{1}{{1 + \beta }}{{\left( {\frac{x}{k}} \right)}^{1 + \beta }} + O\left( {\frac{{{x^\beta }}}{{{k^\beta }}}} \right)} \right)}
& = \frac{{{x^{1 + \beta }}}}{{1 + \beta }}\sum\limits_{k \leqslant x} {\frac{{\mu (k)}}{{{k^{1 + \beta }}}}}  + O\bigg( {{x^\beta }\sum\limits_{k \leqslant x} {\frac{{1}}{{{k^\beta }}}} } \bigg) \nonumber \\
&= \frac{{{x^{1 + \beta }}}}{{(1 + \beta )\zeta (1 + \beta )}} + O({x^\beta }\log^{\lfloor 1/\b\rfloor}x). \label{dbmu}
\end{align}
Let $d$ be the greatest common divisor of $d_1,\dots,d_k$. The first sum on the right-hand side of \eqref{ckbeq} can be written as 
\begin{align}
y\sum_{d\leq x}d^{\beta}\prod_{j=1}^{k}\sum_{\substack{l_jk_j \leq x/d\\(l_1,\dots,l_k)=1}}\mu(k_j)&=y\sum_{d\leq x}d^{\beta}\prod_{j=1}^{k}\sum_{l_jk_j \leq x/d}\mu(k_j)\sum_{l\mid (l_1,\dots,l_k)}\mu(l) \nonumber \\
&=y\sum_{dl\leq x}d^{\beta}\mu(l)\left(\sum_{mn \leq x/dl}\mu(n)\right)^k \nonumber \\
&=\frac{y{{x^{1 + \beta }}}}{{(1 + \beta )\zeta (1 + \beta )}} + O(y{x^\beta }\log^{\lfloor 1/\b\rfloor}x), \nonumber 
\end{align}
where in the last step we used \eqref{dbmu}. The last term on the right-hand side of \eqref{ckbeq} can be estimated as 
\begin{align} \prod_{j=1}^{k}\sum\limits_{{d_j}{k_j} \leqslant x} d_j^{\beta}\ll\bigg(\sum\limits_{{d}{k} \leqslant x} {d^\beta }\bigg)^k\ll_k x^{(1+\beta)k}\log^k x. \nonumber 
\end{align}
This ends the proof of the proposition.
\section{Proof of Theorem \ref{theormsigma}}
If $\beta=1$ and $z=0$, the study of the error term in this asymptotic formula is the well-known Dirichlet divisor problem. Thus, we exclude this case. Let $z=a+ib$, $a\leq 0$, $b \in \mathbb{R}$, and $c=1+1/\log x$. Then by Lemma \ref{montvon} we write 
\begin{align}
\sum_{n\leq x}\sigma_{z,\beta}(n)= \frac{1}{{2\pi i}}\int_{c - iT}^{c + iT} {\zeta (s)\zeta (\beta (s - z))\frac{{{x^s}}}{s}ds}  + E(z, \b;x), \nonumber
\end{align}
where 
\begin{align}
 E(z, \b;x)\ll  \sum_{\substack{x/2 < n < 2x \\ n \ne x}}  \sigma _{a,\beta}(n)\min \bigg( 1, \frac{x}{T|x-n|} \bigg) + \frac{4^{c}+x^{c}}{T} \sum_{n=1}^{\infty} \frac{\sigma _{a,\beta}(n)}{n^{c}}. \nonumber 
\end{align}
From \eqref{sigbz} one has
\begin{align}
\s_{a,\b}(n)\leq \twopartdef{n^{a\b}\s_{0,\beta}(n),}{a>0}{\s_{0,\beta}(n),}{a\leq 0.} \label{sab}
\end{align}
Also we note that
\begin{align}
\sigma_{0,\b}(p)=\twopartdef{2,}{\beta=1}{1,}{\beta\geq 2.}\label{sigmap}
\end{align}
We choose $T=x^{2/3}$. If $0<|x-n|\leq x^{1/3}$, then from \eqref{sab}, \eqref{sigmap}, Lemmas \ref{shiu} and \ref{merest} we have
\begin{align}
\sum_{\substack{0<|x-n|<x^{1/3}}}  \sigma _{a,\beta}(n)\min \bigg( 1, \frac{x}{T|x-n|} \bigg)\ll \sum_{\substack{0<|x-n|<x^{1/3}}}  \sigma _{0,\beta}(n)\ll x^{1/3}\log x.\label{es1}
\end{align} 
For $x+x^{1/3}<n<2x$ one has
\begin{align}
\sum_{\substack{x+x^{1/3}<n<2x}}  \sigma _{a,\beta}(n)\min \bigg( 1, \frac{x}{T|x-n|} \bigg)&\ll\frac{x}{T}\sum_{\substack{x+x^{1/3}<n<2x}}  \frac{\sigma _{0,\beta}(n)}{n-x} \nonumber \\
&\ll \frac{x}{T}\sum_{l\ll \log x}\frac{1}{U}\sum_{\substack{U<n-x<2U\\U=2^{l}x^{1/3}}}\s_{0,\b}(n). \nonumber
 \end{align}
Now by use of Lemmas \ref{shiu} and \ref{merest} we deduce that
\begin{align}
\sum_{\substack{x+x^{1/3}<n<2x}}  \sigma _{a,\beta}(n)\min \bigg( 1, \frac{x}{T|x-n|} \bigg)\ll \frac{x}{T}\log^2x.\label{es2}
\end{align}
The same bound holds when $x/2<n<x-x^{1/3}$. Since 
\begin{align}
\z(\s)\sim \frac{1}{\s-1}\label{zs1}
\end{align}
when $\s\to 1+$, then from Lemma \ref{crum01} we find that
\begin{align}
 \frac{4^{c}+x^{c}}{T} \sum_{n=1}^{\infty} \frac{\sigma _{a,\beta}(n)}{n^{c}}\ll \frac{x}{T}\log ^2x.\label{es3}
\end{align}
Hence from \eqref{es1}, \eqref{es2}, and \eqref{es3} we deduce that
\begin{align}
E(z,\b;x)\ll x^{1/3}\log^2x. \nonumber
\end{align}
Now we take the integral around the rectangle $\mathcal{D} = [-\a-iT,c-iT,c+iT,-\a+iT]$, where $\a=-a+1/\log x$. By the residue theorem one writes
\begin{align}
\frac{1}{2\pi i}\int_{\mathcal{D}}{\zeta (s)\zeta (\beta (s - z))\frac{{{x^s}}}{s}ds}=R, \nonumber 
\end{align}
where $R$ is the sum of the residues at the simple poles $s=1$, $s=z+1/\b$, and $s=0$. The functional equation of $ \z(s) $ is
\begin{align}
\z(s)=\pi^{s-\frac{1}{2}}\frac{\G\left(\frac{1-s}{2}\right)}{\G\left(\frac{s}{2}\right)}\z(1-s):=\chi(s)\z(1-s).\label{fe}
\end{align}
From Stirling's formula for the gamma function \cite[p. 224]{cop} one has
\begin{align}
\chi(\s + it) = \bigg(\frac{2\pi}{t}\bigg)^{\s+it-1/2          }e^{i(t+\pi/4)}\bigg(1+O\bigg(\frac{1}{t}\bigg)\bigg)\label{stirfor}
\end{align}
for fixed $\s$ and $t\geq t_0>0$. Hence by \eqref{fe} we have \begin{align}
\zeta(s) \ll |t|^{\tfrac{1}{2}-\sigma}\label{zblhs}
\end{align}
for $\sigma  < 0$. Also we recall the bound $\z(1+it)\ll \log t/\log\log t$ from \cite[Theorem 5.16]{TitchmarshHeath-Brown}. Therefore from \eqref{fe} and \eqref{stirfor} we have $\z(it)\ll \sqrt{t}\log t/\log \log t$ for $t\geq 2$ . Then by the Phragm\'{e}n-Lindel\"{o}f principle \cite[p.~176]{titch3} one obtains
\begin{align}
\z(\s+it)\ll t^{\frac{1}{2}(1-\s)}\log t/\log \log t\label{zbcritstr}
\end{align}
for $0\leq \s\leq 1$ and $t\geq 2$.
For the upper horizontal integral we have 
\begin{align}
  \int_{ - a + iT}^{c + iT}& {\zeta (s)\zeta (\beta (s - z))\frac{{{x^s}}}{s}ds} \nonumber \\
&=\frac{1}{T}\bigg(\int_{-\a}^{-\frac{1}{2\log x}}+\int_{-\frac{1}{2\log x}}^{a+\frac{1}{\b}+\frac{1}{2\log x}}+\int_{a+\frac{1}{\b}+\frac{1}{2\log x}}^{c}\bigg) \zeta (\s+iT)\zeta (\beta (\s+iT - z)){x^\s}d\s \nonumber \\
&=I_1+I_2+I_3. \nonumber
\end{align}
Using the bound \eqref{zblhs} for $\z(s)$ and \eqref{zbcritstr} for $ \z(\b(s-z)) $ we find 
\begin{align}
I_1\ll \frac{1}{T}\int_{-\a}^{-\frac{1}{2\log x}}T^{\frac{1}{2}-\s}T^{\frac{1}{2}(1-\b \s+\b a)}\log T x^\s d\s\ll x^{a/3}. \nonumber
\end{align}
Using the bound \eqref{zbcritstr} we have 
\begin{align}
I_2\ll \frac{1}{T}\int_{-\frac{1}{2\log x}}^{a+\frac{1}{\b}+\frac{1}{2\log x}}T^{\frac{1}{2}(1-\s)}T^{\frac{1}{2}(1-\b \s+\b a)}\log^2 T x^\s d\s\ll x^{\frac{2-\b+2a\b}{3\b}}\log x. \nonumber 
\end{align}
Similarly 
\begin{align}
I_3\ll \frac{\log ^2T}{T}\int_{a+\frac{1}{\b}+\frac{1}{2\log x}}^{c}T^{\frac{1}{2}(1-\s)}x^\s d\s\ll x^{1/3}\log x. \nonumber
\end{align}
Therefore we obtain 
\begin{align}
 \int_{ - a + iT}^{c + iT}& {\zeta (s)\zeta (\beta (s - z))\frac{{{x^s}}}{s}ds}\ll x^{1/3}\log x. \nonumber
\end{align}
A similar estimate holds for the lower horizontal line. Next we will bound the left vertical part. This is given by
\begin{align}
\int_{- \a - iT}^{-\a + iT} \zeta (s)\zeta (\beta (s - z))\frac{x^s}{s} ds & = \int_{- \a - iT}^{-\a + iT} \chi(s) \chi(\beta (s - z)) \zeta(1-s)\zeta(\beta (1 - s - z))\frac{x^s}{s} ds \nonumber \\
& = \sum\limits_{n = 1}^\infty  \frac{\sigma _{z,\beta }(n)}{n} \int_{- \a - iT}^{-\a + iT} \chi(s) \chi(\beta (s - z))\frac{(xn)^s}{s} ds \nonumber \\
& = i x^{-\a} \sum\limits_{n = 1}^\infty  \frac{\sigma _{z,\beta }(n)}{n^{1+\a}} \int_{-T}^{T} \frac{\chi(-\a + it) \chi(\beta (-\a + it - z))}{-\a + it} (nx)^{it} dt,\label{lhs}
\end{align}
and where in the first step we used the functional equation \eqref{fe}.
Note that 
\begin{align}
 \frac{1}{-\a + it} = \frac{1}{it} + O\left( \frac{1}{t^2} \right).\label{iden}
\end{align}
Applying \eqref{stirfor} and \eqref{iden} yields
\begin{align}
\int_{2b+1}^{T} & \frac{\chi(-\a + it) \chi(\beta (-\a + it - z))}{-\a + it} (nx)^{it} dt \nonumber \\
&\qquad\ll \int_{2b+1}^{T} e^{it(\log(nx)+\log (2\pi e)-\log t)+\b(\log(2\pi e)-\log\b(t-b))}t^{\a+\b\a+\b a}dt. \nonumber
\end{align}
Clearly $\a+\b\a+\b a>0$. Let $F(t):=t(\log(nx)+\log (2\pi e)-\log t)+\b(\log(2\pi e)-\log\b(t-b))$. Then 
\begin{align}
F''(t)=-\frac{1}{t}-\frac{(t-2b)\b}{(t-b)^2}<-\frac{1}{T}-\frac{\b}{T^2}. \nonumber
\end{align}
Therefore by the aid of Lemma \ref{expint} we deduce that
\begin{align}
\int_{2b+1}^{T} & \frac{\chi(-\a + it) \chi(\beta (-\a + it - z))}{-\a + it} (nx)^{it} dt\ll \frac{T^{\a+\b\a+\b a+1}}{\sqrt{T-\b}}. \nonumber
\end{align}
Combining this with \eqref{lhs} we finally get that 
\begin{align}
\int_{- \a - iT}^{-\a + iT} \zeta (s)\zeta (\beta (s - z))\frac{x^s}{s} ds\ll x^{\frac{1+a}{3}}\log^2 x. \nonumber
\end{align}
Now we compute the residues
\begin{align} 
\mathop {\operatorname{res} }\limits_{s = 0} \zeta (s)\zeta (\beta (s - z))\frac{{{x^s}}}{s} &=  - \frac{1}{2}\zeta ( - \beta z), \label{res0} \\ 
\mathop {\operatorname{res} }\limits_{s = 1} \zeta (s)\zeta (\beta (s - z))\frac{{{x^s}}}{s} &= \zeta (\beta (1 - z))x, \nonumber 
\end{align}
and 
\begin{align}
 \mathop {\operatorname{res} }\limits_{s = {\beta ^{ - 1}} + z} \zeta (s)\zeta (\beta (s - z))\frac{{{x^s}}}{s} = \frac{{\zeta ({\beta ^{ - 1}} + z)}}{{1 + \beta z}}{x^{{\beta ^{ - 1}} + z}}.\label{resb}
\end{align} 
Clearly the residue in \eqref{res0} is a constant. When $\b\geq 3$, the residue in \eqref{resb} is absorbed by the error term.
This completes the proof of the theorem.
\section{Proof of Theorem \ref{theormsigma2}}
Let $z_1=a_1+ib_1$, $z_2=a_2+ib_2$, $a_1\leq 0$, $a_2\leq 0$, $a_1+a_2>-1$, $|a_1-a_2|<1/\b$ and $b_1, b_2 \in \mathbb{R}$. Define
\begin{align}
f(z_1,z_2,s;\beta ) := \frac{{\zeta (s)\zeta (\beta (s - z_1))\zeta (\beta (s - z_2))\zeta (\beta (s - z_1 - z_2))}}{{\zeta (\beta (2s - z_1 - z_2))}} \nonumber
\end{align}
and let $c=1+{1/\log x}$.
Then in the view of Lemmas \ref{crum02} and \ref{montvon} we may write 
\begin{align}
\sum\limits_{n \leqslant x} {\sigma _{z_1,\beta}(n)\sigma _{z_2,\beta}(n)}=\frac{1}{{2\pi i}}\int_{c - iT_0 }^{c + iT_0 } {f(z_1,z_2,s;\beta )\frac{{{x^s}}}{s}ds} +E(z_1,z_2,\beta;x),\label{dublesigsum}
\end{align}
where 
\begin{align}
E(z_1,z_2,\beta;x)\ll \sum_{\substack{x/2 < n < 2x \\ n \ne x}}  \sigma _{a_1,\beta}(n)\sigma _{a_2,\beta}(n)\min \bigg( 1, \frac{x}{T_0|x-n|} \bigg) + \frac{4^{c}+x^{c}}{T_0} \sum_{n=1}^{\infty} \frac{\sigma _{a_1,\beta}(n)\sigma _{a_2,\beta}(n)}{n^{c}}. \label{ezzb}
\end{align}
Let $T=2x^{2/3}$ and $T/2<T_0<T$. Now we estimate the right-hand side of \eqref{ezzb}. 
We consider $x+x^{1/3}<n<2x$. Applying \eqref{sab} to the first term of the right-hand side of \eqref{ezzb} we obtain 
\begin{align}
\sum_{x^{1/3}<n-x<x}\sigma _{a_1,\beta}(n)\sigma _{a_2,\beta}(n)\min \bigg( 1, \frac{x}{T_0|x-n|} \bigg)&=\frac{x}{T_0}\sum_{x^{1/3}<n-x<x}\frac{\sigma _{a_1,\beta}(n)\sigma _{a_2,\beta}(n)}{n-x} \nonumber \\
&\leq \frac{x}{T_0}\sum_{0\leq l\ll\log x}\sum_{\substack{U<n-x<2U\\U=2^lx^{1/3}}}\frac{(\sigma _{0,\beta}(n))^2}{n-x} \nonumber \\
&\leq \frac{x}{T_0}\sum_{0\leq l\ll \log x}\frac{1}{U}\sum_{\substack{x+U<n<x+2U\\U=2^lx^{1/3}}}(\sigma _{0,\beta}(n))^2.\label{interest}
\end{align}
From \eqref{sigmap} and Lemmas \ref{shiu} and \ref{merest} we deduce
\begin{align}
\sum_{x+U<n<x+2U}(\sigma _{0,\beta}(n))^2\ll \frac{U}{\log x}\exp(4\log\log x)\ll U\log^3x.\label{shortest}
\end{align}
Invoking this in \eqref{interest} we finally have
\begin{align}
\sum_{x^{1/3}<n-x<x}\sigma _{a_1,\beta}(n)\sigma _{a_2,\beta}(n)\min \bigg( 1, &\frac{x}{T_0|x-n|} \bigg)\ll \frac{x}{T_0}\log^4 x.\label{e2}
\end{align}
Similarly 
\begin{align}
\sum_{x/2<n<x-x^{1/3}}\sigma _{a_1,\beta}(n)\sigma _{a_2,\beta}(n)\min \bigg( 1, &\frac{x}{T_0|x-n|} \bigg)\ll \frac{x}{T_0}\log^4 x.\label{e3}
\end{align}
Let $x-x^{1/3}\leq n\leq  x+x^{1/3}$. Using \eqref{sab} and \eqref{shortest} one has 
\begin{align}
\sum_{0<|x-n|\leq x^{1/3}}\sigma _{a_1,\beta}(n)\sigma _{a_2,\beta}(n)\min \bigg( 1, \frac{x}{T_0|x-n|} \bigg) &\leq \sum_{0<|x-n|\leq x^{1/3}}\sigma _{0,\beta}(n)\sigma _{0,\beta}(n) \nonumber \\
&\ll x^{1/3}\log^3x\label{e1}.
\end{align}
From \eqref{zs1} and Lemma \ref{crum02} we have
\begin{align}
 \frac{4^{c}+x^{c}}{T_0} \sum_{n=1}^{\infty} \frac{\sigma _{a_1,\beta}(n)\sigma _{a_2,\beta}(n)}{n^{c}}\ll \frac{x}{T_0}f(a_1,a_2,c;\b)\ll\frac{x}{T_0}\log^4x ,\label{e4}
\end{align}
for $\b\geq 1$.
Combining \eqref{e2}, \eqref{e3}, \eqref{e1} and \eqref{e4} we obtain 
\begin{align}
E(z_1,z_2,\b;x)\ll  x^{1/3}\log^4x.\label{ezbx}
\end{align}
Let $\l=\tfrac{1}{2}(a_1+a_2+1/\b)$. Suppose $\mathcal{R}$ is a positively oriented contour with vertices $c\pm iT_0$ and $\l\pm iT_0$. By residue theorem we have 
\begin{align}
\frac{1}{{2\pi i}}\int_{\mathcal{R}} {f(z_1,z_2,s;\beta )\frac{{{x^s}}}{s}ds}=R_0,\label{contint}
\end{align}
where $R_0$ is the sum of residues in side the contour $\mathcal{R}$. By H\"older's inequality \cite[p.~382]{titch3} one has 
\begin{align}
\bigg(&\int_{\l-iT_0}^{\l+iT_0}f(z_1,z_2,s;\beta )\frac{{{x^s}}}{s}ds\bigg)^4 \nonumber \\
&\ll\int_{-T_0}^{T_0}\frac{|\zeta(\l+it)|^4x^\l}{|\zeta(\b(2(\l+it)-z_1-z_2))(\l+it)|}dt\int_{-T_0}^{T_0}\frac{|\zeta(\b(\l+it-z_1))|^4x^\l}{|\zeta(\b(2(\l+it)-z_1-z_2))(\l+it)|}dt \nonumber \\
&\quad\times\int_{-T_0}^{T_0}\frac{|\zeta(\b(\l+it-z_2))|^4x^\l}{|\zeta(\b(2(\l+it)-z_1-z_2))(\l+it)|}dt\int_{-T_0}^{T_0}\frac{|\zeta(\b(\l+it-z_1-z_2))|^4x^\l}{|\zeta(\b(2(\l+it)-z_1-z_2))(\l+it)|}dt.\label{lvl}
\end{align} 
By Lemma \ref{ramsan01} we find that 
\begin{align}
 \int_{-T_0}^{T_0}\frac{|\zeta(\b(\l+it-z_1-z_2))|^4x^{\l}}{|\zeta(\b(2(\l+it)-z_1-z_2))(\l+it)|}dt&\ll x^\l (\log^5T_0)\log\log T_0. \nonumber
\end{align}
Let $a_1-a_2\geq 0$. Then by Lemma \ref{ramsan01} we have 
\begin{align}
 \int_{-T_0}^{T_0}\frac{|\zeta(\b(\l+it-z_2))|^4x^{\l}}{|\zeta(\b(2(\l+it)-z_1-z_2))(\l+it)|}dt\ll x^{\l}(\log^5T_0)\log\log T_0. \nonumber
\end{align}
  By the functional equation \eqref{fe} and \eqref{stirfor} one obtains 
\begin{align}
 \int_{-T_0}^{T_0}&\frac{|\zeta(\b(\l+it-z_1))|^4}{|\zeta(\b(2(\l+it)-z_1-z_2))(\l+it)|}dt \nonumber \\
&=\int_{-T_0}^{T_0}|\chi(\b(\l+it-z_1))|\frac{|\zeta(1-\b(\l+it-z_1))|^4}{|\zeta(\b(2(\l+it)-z_1-z_2))(\l+it)|}dt \nonumber \\
&\ll T_0^{(a_1-a_2)\b/2}\int_{-T_0}^{T_0}\frac{|\zeta(1/2+(a_1-a_2)\b/2+i(t+t'))|^4}{|\zeta(1+2it))(\l+i(t+t''))|}dt, \nonumber
\end{align}
where in last step we made a suitable change of variable. Finally, by Lemma \ref{ramsan01} we obtain 
\begin{align}
 \int_{-T_0}^{T_0}&\frac{\zeta(\b(\l+it-z_1))|^4x^\l}{|\zeta(\b(2(\l+it)-z_1-z_2))(\l+it)|}dt\ll x^\l T_0^{(a_1-a_2)\b/2}(\log^5T_0)\log\log T_0. \nonumber
\end{align}
The case $a_1-a_2\leq 0$ can be treated similarly. Therefore for any sign of $a_1-a_2$ we have 
\begin{align}
 \int_{-T_0}^{T_0}&\frac{\zeta(\b(\l+it-z_i))|^4x^\l}{|\zeta(\b(2(\l+it)-z_1-z_2))(\l+it)|}dt\ll x^\l T_0^{|a_1-a_2|\b/2}(\log^5T_0)\log\log T_0, \nonumber
\end{align}
for $i=1,2$.
Using a similar argument one can deduce that
\begin{align}
 \int_{-T_0}^{T_0}\frac{|\zeta(\l+it)|^4x^{\l}}{|\zeta(\b(2(\l+it)-z_1-z_2))(\l+it)|}dt\ll x^{\l}T_0^{1/2-\l}(\log^5T_0)\log\log T_0 . \nonumber
\end{align}
Thus from \eqref{lvl} we have 
\begin{align}
\int_{\l-iT_0}^{\l+iT_0}f(z_1,z_2,s;\beta )\frac{{{x^s}}}{s}ds\ll \max\bigg(x^{\tfrac{1}{2\b}+\tfrac{a_1+a_2}{2}+\tfrac{\b|a_1-a_2|}{3}},x^{\tfrac{1}{3}+\tfrac{1}{6\b}+\tfrac{a_1+a_2}{6}}\bigg)(\log^5x)\log\log x .\label{leftvertline}
\end{align}
Next we compute the integral
\begin{align}
\int_{\l}^{c}\int_{T/2}^{T}f(z_1,z_2,\s+it;\beta )\frac{{{x^{\s+it}}}}{\s+it}d\s dt. \nonumber
\end{align}
Using the functional equation \eqref{fe} we write 
\begin{align}
&\int_{\l}^{c}\int_{T/2}^{T}\frac{|\zeta(\s+it)|^4x^{\s}}{|\zeta(\b(2(\s+it)-z_1-z_2))(\s+it)|}d\s dt \nonumber \\
&=\int_{\l}^{1/2}\int_{T/2}^{T}\frac{|\chi(\s+it)||\zeta(1-\s-it)|^4x^{\s}}{|\zeta(\b(2(\s+it)-z_1-z_2))(\s+it)|}d\s dt \nonumber \\
&\qquad+\int_{1/2}^{c}\int_{T/2}^{T}\frac{|\zeta(\s+it)|^4x^{\s}}{|\zeta(\b(2(\s+it)-z_1-z_2))(\s+it)|}d\s dt. \nonumber
\end{align}
Now by the aid of \eqref{stirfor} and Lemma \ref{ramsan02} we deduce that
\begin{align}
\int_{\l}^{c}\int_{T/2}^{T}\frac{|\zeta(\s+it)|^4x^{\s}}{|\zeta(\b(2(\s+it)-z_1-z_2))(\s+it)|}d\s dt\ll \frac{x}{\log x}(\log^4T)\log\log T. \nonumber
\end{align}
If $a_1-a_2\geq 0$, then similarly we can find that
\begin{align}
\int_{\l}^{c}\int_{T/2}^{T}\frac{|\zeta(\b(\s+it-z_1)|^4x^{\s}}{|\zeta(\b(2(\s+it)-z_1-z_2))(\s+it)|}d\s dt\ll \frac{x}{\log x}(\log^4T)\log\log T. \nonumber
\end{align}
From Lemma \ref{ramsan02} we have 
\begin{align}
\int_{\l}^{c}\int_{T/2}^{T}\frac{|\zeta(\b(\s+it-z_2)|^4x^{\s}}{|\zeta(\b(2(\s+it)-z_1-z_2))(\s+it)|}d\s dt\ll \frac{x}{\log x}(\log^4T)\log\log T \nonumber
\end{align}
and
\begin{align}
\int_{\l}^{c}\int_{T/2}^{T}\frac{|\zeta(\b(\s+it-z_1-z_2)|^4x^{\s}}{|\zeta(\b(2(\s+it)-z_1-z_2))(\s+it)|}d\s dt\ll \frac{x}{\log x}(\log^4T)\log\log T. \nonumber
\end{align}
Therefore by H\"older's inequality
\begin{align}
\int_{\l}^{c}\int_{T/2}^{T}f(z_1,z_2,\s+it;\beta )\frac{{{x^{\s+it}}}}{\s+it}d\s dt\ll \frac{x}{\log x}(\log^4T)\log\log T. \nonumber
\end{align}
Hence we can choose a suitable $T_0$ so that $T/2\leq T_0\leq T$ and 
\begin{align}
\int_{\l}^{c}f(z_1,z_2,\s+iT_0;\beta )\frac{{{x^{\s+iT_0}}}}{\s+iT_0}d\s&\ll \frac{x}{T_0\log x}(\log^4T_0)\log\log T_0 \nonumber \\&\ll x^{1/3}(\log^3 x)\log\log x.\label{horline}
\end{align}
Finally combining \eqref{dublesigsum}, \eqref{ezbx}, \eqref{contint}, \eqref{leftvertline} and \eqref{horline} we find 
\begin{align}
\sum\limits_{n \leqslant x} {\sigma _{z_1,\beta}(n)\sigma _{z_2,\beta}(n)}=R_0+O\bigg(\max\bigg(x^{\tfrac{1}{2\b}+\tfrac{a_1+a_2}{2}+\tfrac{\b|a_1-a_2|}{3}},x^{\tfrac{1}{3}+\tfrac{1}{6\b}+\tfrac{a_1+a_2}{6}}\bigg)(\log^5x)\log\log x\bigg).\label{dublesum1}
\end{align}
Since $z_i\neq 0$, $z_1\neq z_2$, and $|\Re(z_1-z_2)|<1/\b$, then all the poles are simple. The residues at the simple poles $s=1$, $s=z_1+1/\b$, $s=z_2+1/\b$, and $s=z_1+z_2+1/\b$ are given by
\begin{align}
  \mathop {\operatorname{res} }\limits_{s = 1} f(z_1,z_2,s;\beta )\frac{{{x^s}}}{s} &= \frac{{\zeta (\beta (1 - z_1))\zeta (\beta (1 - z_2))\zeta (\beta (1 - z_1 - z_2))}}{{\zeta (\beta (2 - z_1 - z_2))}}x, \nonumber \\ 
  \mathop {\operatorname{res} }\limits_{s = z_1 + 1/\beta } f(z_1,z_2,s;\beta )\frac{{{x^s}}}{s} &= \frac{{\zeta (z_1 + \tfrac{1}{\beta })\zeta (1 + \beta z_1 - \beta z_2)\zeta (1 - \beta z_2)}}{{(z_1\beta  + 1)\zeta (2 + \beta z_1 - \beta z_2)}}{x^{z_1 + \tfrac{1}{\beta }}}, \label{erterm1} \\
  \mathop {\operatorname{res} }\limits_{s = z_2 + 1/\beta } f(z_1,z_2,s;\beta )\frac{{{x^s}}}{s} &= \frac{{\zeta (z_2 + \tfrac{1}{\beta })\zeta (1 + \beta z_2 - \beta z_1)\zeta (1 - \beta z_1)}}{{(\beta z_2 + 1)\zeta (2 + \beta z_2  - \beta z_1)}}{x^{z_2 + \tfrac{1}{\beta }}},\label{erterm2}
\end{align}
and
\begin{align} 
\mathop {\operatorname{res} }\limits_{s = z_1 + z_2 + 1/\beta } f(z_1,z_2,s;\beta )\frac{{{x^s}}}{s} = \frac{{\zeta (z_1 + z_2 + \tfrac{1}{\beta })\zeta (\beta z_2 + 1)\zeta (\beta z_1 + 1)}}{{(\beta z_1 + \beta z_2 + 1)\zeta (2 + \beta z_1 + \beta z_2)}}{x^{z_1 + z_2 + \tfrac{1}{\beta }}}.\label{erterm3}
\end{align}
If $ \b\geq 3 $, then \eqref{erterm1}, \eqref{erterm2}, and \eqref{erterm3} are smaller than the error term of the right-hand side of \eqref{dublesum1}. Therefore for $\b\geq 3$ we find 
\begin{align}
R_0= \frac{{\zeta (\beta (1 - z_1))\zeta (\beta (1 - z_2))\zeta (\beta (1 - z_1 - z_2))}}{{\zeta (\beta (2 - z_1 - z_2))}}x. \nonumber
\end{align}
For $\b=1$ and $-1/2<\Re(z_1),\Re(z_2), \Re(z_1+z_2)<0$ we find 
\begin{align}
R_0=x\bigg(&\frac{{\zeta (1 - z_1)\zeta (1 - z_2)\zeta (1 - z_1 - z_2)}}{{\zeta  (2 - z_1 - z_2)}}+\frac{{\zeta (z_1 + 1)\zeta (1 +  z_1 -  z_2)\zeta (1 - z_2)}}{{(z_1 + 1)\zeta (2 +  z_1 - z_2)}}x^{z_1} \nonumber \\
&\quad+\frac{{\zeta (z_2 + 1)\zeta (1 + z_2 - z_1)\zeta (1 - z_1)}}{{(z_2 + 1)\zeta (2 + z_2  -  z_1)}}x^{z_2}+\frac{{\zeta (z_1 + z_2 + 1)\zeta ( z_2 + 1)\zeta ( z_1 + 1)}}{{( z_1 +  z_2 + 1)\zeta (2 +  z_1 +  z_2)}}x^{z_1+z_2}\bigg). \nonumber
\end{align}
Finally for $\b=2$ and $-1/10<\Re(z_1),\Re(z_2),\Re(z_1+z_2)<0$ we obtain 
\begin{align}
R_0=&\sqrt{x}\bigg(\frac{{\zeta (2 (1 - z_1))\zeta (2 (1 - z_2))\zeta (2 (1 - z_1 - z_2))}}{{\zeta (2 (2 - z_1 - z_2))}}\sqrt{x}+\frac{{\zeta (z_1 + \tfrac{1}{2 })\zeta (1 + 2 z_1 - 2 z_2)\zeta (1 - 2 z_2)}}{{(2z_1  + 1)\zeta (2 + 2 z_1 - 2 z_2)}}{x^{z_1}} \nonumber \\
&\qquad+\frac{{\zeta (z_2 + \tfrac{1}{2 })\zeta (1 + 2 z_2 - 2 z_1)\zeta (1 - 2 z_1)}}{{(2 z_2 + 1)\zeta (2 + 2 z_2  - 2 z_1)}}{x^{z_2 }}+ \frac{{\zeta (z_1 + z_2 + \tfrac{1}{2 })\zeta (2 z_2 + 1)\zeta (2 z_1 + 1)}}{{(2 z_1 + 2 z_2 + 1)\zeta (2 + 2 z_1 + 2 z_2)}}{x^{z_1 + z_2}}\bigg). \nonumber
\end{align}
This completes the proof of the theorem.
\section{Proof of Theorem \ref{maintheorem}}
Let us consider 
\begin{align}
\alpha=1+\frac{1}{\log y}, \nonumber
\end{align} 
 $y\geq x$, and $T=y^{2/3}$.
By Lemma \ref{montvon} one finds 
\begin{align}
\sum\limits_{q^\b \le x} {c_{q,\beta}(n)}  = \frac{1}{{2\pi i}}\int_{\alpha  - iT}^{\alpha  + iT} {\frac{{\sigma _{1 - s,\beta }(n)}}{{\zeta (\b s)}}\frac{{{x^s}}}{s}ds}  + {E_1}(x,n),\label{cqbn}
\end{align}
where 
\begin{align}
 {E_1}(x,n)\ll\sum_{\substack{x/2<q^\b<2x\\q^\b\neq x}}| {c_{q,\beta }(n)}|\min\left(1,\frac{x}{T|x-q^\b|}\right)+\frac{x^{\alpha}}{T}\sum_{q=1}^{\infty}\frac{| {c_{q,\beta}(n)}|}{q^{\b\alpha}}.\label{erre1}
\end{align}
Using Lemma \ref{cohen02}, we have
\begin{align}
\sum_{q=1}^{\infty}\frac{| {c_{q,\beta}(n)}|}{q^{\b\alpha}}=\sum_{q=1}^{\infty}\frac{1}{q^{\b\a}}\bigg|\sum_{\substack{d|q \\ {d^\beta }|n}} {{d^\beta }\mu \left( {\frac{q}{d}} \right)}\bigg|\leq \sum_{q=1}^{\infty}\frac{1}{q^{\b\a}}\sum_{\substack{ {d^\beta }|n}} {{d^{\b-\b\a} }}=\s_{-\frac{1}{\log y},\b}(n)\z(\b\a). \nonumber
\end{align}
Therefore from \eqref{zs1} we deduce 
\begin{align}
\sum_{q=1}^{\infty}\frac{| {c_{q,\beta}(n)}|}{q^{\alpha}}\ll\twopartdef{\s_{-\frac{1}{\log y},\b}(n)\log y,}{\b>1,}{\s_{-\frac{1}{\log y},\b}(n),}{\b=1.}\label{erst}
\end{align} 
Similarly 
\begin{align}
\sum_{\substack{x/2<q^\b<2x\\q^\b\neq x}}| {c_{q,\beta }(n)}|\min\left(1,\frac{x}{T|x-q^\b|}\right)&\leq \sum_{\substack{ {d^\beta }|n}} {{d^{\b} }}\sum_{\substack{x/2<d^\b q^\b<2x\\d^b q^\b\neq x}}\min\left(1,\frac{x}{T|x-d^\b q^\b|}\right) \nonumber \\
&\ll \frac{x}{T}\s_{0,\b}(n)\log x. \label{coramest} 
\end{align}
Hence from \eqref{erre1}, \eqref{erst} and \eqref{coramest} we obtain
\begin{align}
{E_1}(x,n)&\ll\frac{x}{T}\s_{0,\b}(n)\log y.\label{estfe1}
\end{align}
Now by \eqref{sigmap} and Lemma \ref{shiu} one has 
\begin{align}
\sum_{n\leq y}E_1(x,n)\ll \twopartdef{\frac{xy}{T}\log y,}{\b>1,}{\frac{xy}{T}\log^2 y,}{\b=1.} \nonumber 
\end{align} 
Summing both sides of \eqref{cqbn} over $n$ and using Theorem \ref{theormsigma} we can write 
\begin{align}
C_{1,\beta}(x,y)= \sum\limits_{n \leqslant y} {\sum\limits_{q^\b \leqslant x} {c_{q,\beta }(n)} }&=\frac{y}{{2\pi i}}\int_{\alpha  - iT}^{\alpha  + iT} {\frac{{{x^s}}}{s}ds}+\frac{y^{1+1/\b}}{{2\pi i}}\int_{\alpha  - iT}^{\alpha  + iT} {\frac{\z(1-s+1/\b)}{1+\b-\b s}\frac{{{(x/y)^s}}}{s\z(\b s)}ds} \nonumber \\
&+O\bigg(y^{1/3}\log^2y\int_{2}^{T}\frac{x^\a}{t|\z(\b(\a+it))|}dt\bigg)  + \sum_{n\leq y}{E_1}(x,n).\label{c1bxy} 
\end{align}
Note that $ (\z(\s+it))^{\pm 1}\ll\log t  $ for $1\leq \s\le 1/\log y$. Thus the third term in the right-hand side of \eqref{c1bxy} is 
\begin{align}
\ll xy^{1/3}\log^2y\log^2T. \nonumber
\end{align}
By \eqref{fourpartiden} the first integral in the right-hand side of \eqref{c1bxy} is 
\begin{align}
y+O\bigg(\frac{yx}{T}\bigg) \nonumber
\end{align}
for $x\geq 2$.
For the second integral we shift the line of integration from $\s=\a$ to $\s=1+\a+1/\b$. The residue due to the simple pole at $s=1+1/\b$ is 
\begin{align}
\mathop {\operatorname{res} }\limits_{s=1+1/\b} \frac{\z(1-s+1/\b)}{1+\b-\b s}\frac{{{y^{1+1/\b}(x/y)^s}}}{s\z(\b s)}=\frac{x^{1+1/\b}}{2(1+\b)\z(1+\b)}. \nonumber
\end{align}
The contribution from the horizontal line is 
\begin{align}
&\ll \frac{y^{1+\frac{1}{\b}}}{T^2}\bigg(\log T\int_{\a}^{\a+\frac{1}{\b}}T^{\frac{1}{2}(\s-\frac{1}{\b})}\bigg(\frac{x}{y}\bigg)^{\s}d\s+\int_{\a+\frac{1}{\b}}^{1+\a+\frac{1}{\b}}T^{-\frac{1}{2}+\s-\frac{1}{\b}}\bigg(\frac{x}{y}\bigg)^{\s}d\s\bigg) \nonumber \\
&\ll \frac{xy^{1/\b}}{T^{3/2}}\log T+\frac{x^{1+1/\b}}{T^{1/2}}\log T \nonumber
\end{align}
Similarly the contribution from right vertical line is 
\begin{align}
\ll\frac{x^{2+1/\b}}{y^{\a}}\int_{2}^{T}t^{-\frac{1}{2}+\frac{1}{\log y}}dt\ll \frac{x^{2+1/\b}}{y^{\a}}T^{\frac{1}{2}+\frac{1}{\log y}}. \nonumber
\end{align}
Note that the second integral of \eqref{c1bxy} disappears when $\b\ge 3$. Finally, replace $x$ by $x^\b$ to end the proof.
\section{Proof of Theorem \ref{maintheorem02}}
For $j \in \{1,2\}$, we let $\alpha_j$ be such that
\begin{align}{\alpha _j}= 1  + \frac{j}{{\log y}}. \nonumber \end{align}
Let $y\ge x$ and $T=x^2\log^5x$. From \eqref{cqbn} and \eqref{estfe1} we have
\begin{align}
\sum\limits_{q^ \b\leqslant x} {c_{q,\beta}(n)}  = \frac{1}{{2\pi i}}\int_{\alpha_j  - iT}^{\alpha_j  + iT} {\frac{{\sigma _{1 - s,\beta }(n)}}{{\zeta (\b s)}}\frac{{{x^s}}}{s}ds}  + O\bigg(\frac{x}{T}\s_{0,\b}(n)\log y\bigg). \nonumber
\end{align}
Note that 
\begin{align}
\frac{1}{{2\pi i}}\int_{\alpha_j  - iT}^{\alpha_j  + iT} {\frac{{\sigma _{1 - s,\beta }(n)}}{{\zeta (\b s)}}\frac{{{x^s}}}{s}ds}\ll x\s_{0,\b}(n)\log^2 T. \nonumber
\end{align}
Therefore
\begin{align}
  {\bigg( {\sum\limits_{q^b \leqslant x} {c_{q,\beta}(n)} } \bigg)^2} &= \frac{1}{{{{(2\pi i)}^2}}}\int_{{\alpha _1} - iT}^{{\alpha _1} + iT} {\int_{{\alpha _2} - 2iT}^{{\alpha _2} + 2iT} {\frac{{\sigma _{1 - {s_1},\beta}(n)}}{{\zeta ({\b s_1})}}\frac{{\sigma _{1 - {s_2},\beta}(n)}}{{\zeta ({\b s_2})}}\frac{{{x^{{s_1} + {s_2}}}}}{{{s_1}{s_2}}}d{s_2}} d{s_1}} \nonumber  \\
   &+ O\left( {\frac{{{x^{2}}}}{T}{{(\sigma _{0,\beta }(n))}^2}\log y\log^2 T} \right). \label{cb2} 
\end{align}
Combining equation \eqref{sigmap}, Lemmas \ref{shiu} and \ref{merest} we find
\begin{align}
\sum_{n<y}(\sigma _{0,\beta}(n))^2\ll  y\log^3y.\label{shortest}
\end{align}
Now sum over $n$ both sides of \eqref{cb2} so that
\begin{align}
  C_{2,\beta}(x,y) = \sum\limits_{n \leqslant y}{{{\bigg( {\sum\limits_{q \leqslant x} {c_{q,\beta}(n)} } \bigg)}^2}} &= \frac{1}{{{{(2\pi i)}^2}}}\int_{{\alpha _1} - iT}^{{\alpha _1} + iT} {\int_{{\alpha _2} - 2iT}^{{\alpha _2} + 2iT} {\frac{{G({s_1},{s_2},\b,n)}}{{\zeta ({s_1})\zeta ({s_2})}}\frac{{{x^{{s_1} + {s_2}}}}}{{{s_1}{s_2}}}d{s_2}} d{s_1}} \nonumber  \\
   &+ O\left( \frac{x^{2}y}{T}\log^4 y\log^2 T \right) \nonumber \\
&=I+  O\left( \frac{x^{2}y}{T}\log^4 y\log^2 T \right), \nonumber
\end{align}
where
\begin{align}
G({s_1},{s_2},\b ,y) = \sum\limits_{n \leqslant y} {\s_{1 - {s_1},\beta}(n)\sigma _{1 - {s_2},\beta}(n)}.  \nonumber
\end{align}
From Theorem \ref{theormsigma2} we find
\begin{align}
I=I_1+I_2+I_3+I_4+O\bigg(x^2y^{\frac{1}{3}+\frac{1}{6\b}}(\log^5y)(\log^4T)\log\log y\bigg), \nonumber
\end{align}
where
\begin{align}
I_1&=\frac{y}{(2\pi i)^2}\int_{\a_1 - iT}^{\a_1+ iT} \int_{\a_2 - 2iT}^{\a_2 + 2iT}\frac{\z(\b(s_1+s_2-1))}{\z(\b(s_1+s_2))}\frac{x^{s_1+s_2}}{s_1s_2}ds_1ds_2, \nonumber \\
I_2&=\frac{y^{1+1/\b}}{(2\pi i)^2}\int_{\a_1 - iT}^{\a_1+ iT} \int_{\a_2 - 2iT}^{\a_2 + 2iT}\frac{\z(1-s_1+1/\b)\z(1-\b s_1+\b s_2)\z(1-\b+\b s_2))}{(1+\b-\b s_1)\z(2-\b s_1+\b s_2)\z(\b s_1)\z(\b s_2)}\frac{y^{-s_1}x^{s_1+s_2}}{s_1s_2}ds_1ds_2, \nonumber \\
I_3&=\frac{y^{1+1/\b}}{(2\pi i)^2}\int_{\a_1 - iT}^{\a_1+ iT} \int_{\a_2 - 2iT}^{\a_2 + 2iT}\frac{\z(1-s_2+1/\b)\z(1-\b s_2+\b s_1)\z(1-\b+\b s_1))}{(1+\b-\b s_2)\z(2-\b s_2+\b s_1)\z(\b s_1)\z(\b s_2)}\frac{y^{-s_2}x^{s_1+s_2}}{s_1s_2}ds_1ds_2, \nonumber \\
\text{and} \nonumber \\
I_4&=\frac{y^{2+1/\b}}{(2\pi i)^2}\int_{\a_1 - iT}^{\a_1+ iT} \int_{\a_2 - 2iT}^{\a_2 + 2iT}\frac{\z(2-s_1-s_2+1/\b)\z(1+\b-\b s_1)\z(1+\b-\b s_2))}{(1+2\b-\b s_1-\b s_2)\z(2+2\b-\b s_2-\b s_1)\z(\b s_1)\z(\b s_2)} \nonumber \\
&\hspace{10cm}\times\frac{(x/y)^{s_1+s_2}}{s_1s_2}ds_1ds_2. \nonumber
\end{align}
Note that the integrals $I_2$, $I_3$, and $I_4$ disappear when $\b\ge 3$. First we will compute the integral $I_1$. Let 
\begin{align}
J_1(s_1)=\frac{1}{2\pi i}\int_{\a_2 - iT}^{\a_2+ iT} \frac{\z(\b(s_1+s_2-1))}{\z(\b(s_1+s_2))}\frac{x^{s_2}}{s_2}ds_2. \nonumber
\end{align}
Shift the line of integration from $\s=\a_2$ to $\s=1+\tfrac{1}{2\b}-\a_1$. Note that the integrand has a simple pole at $1+1/\b-s_1$ in this region. The residue is 
\begin{align}
\mathop {\operatorname{res} }\limits_{s_2=1-s_1+1/\b}\frac{\z(\b(s_1+s_2-1))}{\z(\b(s_1+s_2))}\frac{x^{s_2}}{s_2}=\frac{x^{1-s_1+1/\b}}{(1+\b-\b s_1)\z(1+\b)}. \nonumber
\end{align}
Let $T\geq x^{2/\b}$. The contribution from the horizontal line is 
\begin{align}
\ll \frac{1}{T^{\frac{1}{2}}}\int_{1-\a_1+1/2\b}^{1/\b}\bigg(\frac{x}{T^{\b/2}}\bigg)^\s d\s+\frac{\log x}{T}\int_{1/\b}^{\a_2}(x)^\s d\s\ll \frac{x}{T}+\frac{x^{1/2\b}}{T^{3/4}}. \nonumber
\end{align}
Hence 
\begin{align}
I_1=&\frac{y}{(2\pi i)^2}\int_{\a_1 - iT}^{\a_1+ iT} \int_{\frac{1}{2\b}-\frac{1}{\log y} - 2iT}^{\frac{1}{2\b}-\frac{1}{\log y} + 2iT}\frac{\z(\b(s_1+s_2-1))}{\z(\b(s_1+s_2))}\frac{x^{s_1+s_2}}{s_1s_2}ds_1ds_2 \nonumber \\
&+\frac{y}{2\pi i}\int_{\a_1-iT}^{\a_1+iT}\frac{x^{1+1/\b}}{s_1(1+\b-\b s_1)\z(1+\b)}ds_1+O\bigg(  \frac{yx^2}{T}\log x+\frac{yx^{1+1/2\b}}{T^{3/4}}\log x\bigg).\label{I1}
\end{align}
Denote the first integral in the right-hand side of \eqref{I1} by $I_{11}$. Then we have
\begin{align}
I_{11}&\ll yx^{1+1/2\b}\int_{1}^{T}\int_{1}^{2T}\frac{|\z(1/2+i\b(t_1+t_2))|}{t_1t_2}dt_1dt_2 \nonumber \\
&\ll yx^{1+1/2\b}\log^2T\int_{T/2}^{T}\int_{T}^{2T}\frac{|\z(1/2+i\b(t_1+t_2))|}{t_1t_2}dt_1dt_2 \nonumber \\
&\ll yx^{1+1/2\b}\log^2T\bigg(\int_{T/2}^{T}\int_{T}^{2T}|\z(1/2+i\b(t_1+t_2))|^2dt_1dt_2\times\int_{T/2}^{T}\int_{T}^{2T}\frac{1}{t_1^2t_2^2}dt_1dt_2\bigg)^{1/2}, \nonumber
\end{align}
where in the last step we used H\"older's inequality. By the aid of the mean value theorem of $\z(s)$ \cite[Theorem 7.3]{TitchmarshHeath-Brown} we deduce that
\begin{align}
I_{11}\ll yx^{1+1/2\b}\log^3T. \nonumber
\end{align}
Finally, by applying the residue theorem on the second integral in the right-hand side of \eqref{I1} we conclude that
\begin{align}
I_1=\frac{yx^{1+1/\b}}{(1+\b)\z(1+\b)}+O\bigg(yx^{1+1/2\b}\log^3T+\frac{yx^2}{T}\log x+\frac{yx^{1+1/2\b}}{T^{3/4}}\log x\bigg). \nonumber
\end{align}
Next we compute the integral $I_2$. Let
\begin{align}
J_2(s_2)=\frac{1}{2\pi i}\int_{\a_1 - iT}^{\a_1+ iT} \frac{\z(1-s_1+1/\b)\z(1-\b s_1+\b s_2)}{(1+\b-\b s_1)\z(2-\b s_1+\b s_2)\z(\b s_1)}\frac{y^{-s_1}x^{s_1}}{s_1}ds_1 \nonumber
\end{align}
Shift the line of integration in the $s_1$-plane from $\s=\a_1$ to $\s=\a_3=1+\tfrac{1}{\b}-\tfrac{3}{\log y}$. Note that $s_1=s_2$ is a simple pole of the integrand and the residue is 
\begin{align}
\mathop {\operatorname{res} }\limits_{s_1=s_2}\frac{\z(1-s_1+1/\b)\z(1-\b s_1+\b s_2)}{(1+\b-\b s_1)\z(2-\b s_1+\b s_2)\z(\b s_1)}\frac{y^{-s_1}x^{s_1}}{s_1}=-\frac{\z(1-s_2+1/\b)}{\b\z(2)(1+\b-\b s_2)\z(\b s_2)}\frac{(x/y)^{s_2}}{s_2}. \nonumber
\end{align}
The contribution from the horizontal lines is 
\begin{align}
\ll \frac{x^{1+1/\b}}{y^{1+1/\b}T}\log T. \nonumber
\end{align}
provided $x\le y<x^{\b+2}\log^{\frac{5}{2}(\b+1)}y$.
The contribution from the vertical line $\s=\a_3$ is
\begin{align}
\ll \frac{x^{1+1/\b}}{y^{1+1/\b}}\log T. \nonumber
\end{align}
Therefore 
\begin{align}
I_2=\frac{y^{1+1/\b}}{\b\z(2)}\frac{1}{2\pi i}\int_{\a_2-2iT}^{\a_2+2iT}&\frac{\z(1-s_2+1/\b)\z(1-\b+\b s_2)}{(1+\b-\b s_2)\z^2(\b s_2)}\frac{(x^2/y)^{s_2}}{s_2^2}ds_2 \nonumber \\
&+O\bigg(\frac{x^{2+1/\b}}{T}\log^2 T+x^{2+1/\b}\log^2 T\bigg). \nonumber
\end{align}
Next we shift the line of integration in the $s_2$-plane form $\s=\a_2$ to $\s=\a_4=1+\a_2+\tfrac{1}{\b}$. Note that $s_2=1+\tfrac{1}{\b}$ is a simple pole and the residue is 
\begin{align}
\mathop {\operatorname{res} }\limits_{s_2=1+1/\b}\frac{y^{1+1/\b}\z(1-s_2+1/\b)\z(1-\b+\b s_2)}{\b\z(2)(1+\b-\b s_2)\z^2(\b s_2)}\frac{(x^2/y)^{s_2}}{s_2^2}=\frac{x^{2+2/\b}}{2(1+\b)^2\z^2(1+\b)}. \nonumber
\end{align}
If we split the interval $(\a_2,1+\a_2+1/\b)$ into two subintervals $(\a_2,1+1/b)$ and $(1+1/b,1+a_2+1/\b)$, then the horizontal line integration is 
\begin{align}
\ll(y^{-1}x^{2})^{1+1/\b}T^{-5/2}+(y^{-1}x^{2})^{2+1/\b}T^{-2}. \nonumber
\end{align}
The vertical line integration is 
\begin{align}
\ll y^{-1}x^{4+2/\b}. \nonumber
\end{align}
To bound the integral $I_3$ we move the line of integration in $s_2$ plane from $\s=\a_2$ to $\s=\a_2+\tfrac{1}{\b}$. The contribution from the horizontal line is 
\begin{align}
\ll \frac{xy^{1+1/\b}}{T^2}\log^2 T\int_{\a_2}^{a_2+1/\b}T^{\frac{1}{2}(\b\s+\s-1/\b)}(x/y)^{\s}d\s\ll \frac{x^{2+1/\b}}{T}\log^2T, \nonumber
\end{align} 
provided that $x\le y<x^{\b+2}\log^{\frac{5}{2}(\b+1)}y$. If $\a_5=\a_2+1/\b$, then the contribution from the left vertical line is 
\begin{align}
\ll x^{2+1/\b}\int_{2}^{T}\int_{2}^{T}\frac{\sqrt{t_2(t_2-t_1)}}{t_2^2t_1}dt_2dt_1\ll x^{2+1/\b}\log^2 T. \nonumber
\end{align} 
Similarly if one moves the line of integration in $s_2$-plane from $\s=\a_2$ to $\s=1+\tfrac{1}{\b}-\frac{4}{\log x}$, then it can be shown that
\begin{align}
I_4\ll x^{2+1/\b}\log^2 T. \nonumber
\end{align}
Now we complete the proof of the theorem by replacing $x$ by $x^\b$.
\section*{Appendix}
It is worth remarking (see \cite{ChanKumchev}) that the introduction of van der Corput's method of exponential sums leads to the following result concerning the first moment.
\begin{theorem} \label{appendixtheorem}
Let $\beta \in \mathbb{N}$ be fixed. 
 Let $x$ be a large real and $y \ge x^{\beta}$. One has
\begin{align} 
C_{1,\beta }(x,y) = y - \frac{{{x^{1 + \beta }}}}{{2(1 + \beta )\zeta (1 + \beta )}} + R_{1,\beta}(x,y), \nonumber 
\end{align}
where 
\begin{equation}
R_{1,\beta}(x,y)\ll_{\beta}
\begin{cases}
xy^{\frac{1}{3}}\log x + x^3y^{ - 1}, & \mbox{ if } \beta=1,\\
x^{\frac{1 + 2\beta }{3}}y^{\frac{1}{3}} + x^{1 + 2\beta }y^{ - 1}, & \mbox{ if } \beta>1. \nonumber
\end{cases}
\end{equation}
\end{theorem}
We remark that the range of $y$ is different than the one in Theorem \ref{maintheorem} and that when $\beta=1$, \eqref{chankumchev1} follows as a special case.
\\
To prove this, we recall the following 
auxiliary lemma from \cite[Lemma 4.3]{GrahamKolesnikvanderCorput}.
For the definition of an exponent pair, the reader is referred to \cite[pp. 30-31]{GrahamKolesnikvanderCorput}.
\begin{lemma} \label{exppairlemma}
Suppose that $(k,l)$ is an exponent pair and $\mathbf{I}$ is a subinterval of $(N,2N]$, then
\[\sum\limits_{n \in \mathbf{I}} {\psi (y/n)}  \ll {y^{k/(k + 1)}}{N^{((l-k)/(k + 1)}} + {y^{ - 1}}{N^2},\]
where $\psi(t) := t - [t] - \tfrac{1}{2}$ denotes the saw-tooth function. Here $[t]$ stands for the integral part of $t$. 
\end{lemma}
By the use of \eqref{cohenmoebius} we have 
\[C_{1,\beta}(x,y) = \sum\limits_{n \leqslant y} {\sum\limits_{q \leqslant x} {c_q^{(\beta )}(n)} }  = \sum\limits_{n \leqslant y} {\sum\limits_{q \leqslant x} {\sum_{\substack{d|q \\ {d^\beta |n }}} {{d^\beta }\mu \left( {\frac{q}{d}} \right)} } }  = \sum\limits_{n \leqslant y} {\sum_{\substack{dk \leqslant x \\ d^\beta |n}} {{d^\beta }\mu (k)} } \]
where we have made the change $k= \tfrac{q}{d}$. Interchanging the order of summation we obtain
\begin{align}
  C_{1,\beta}(x,y) &= \sum\limits_{dk \leqslant x} {d^\beta }\mu (k)\sum_{\substack{n \le y \\ d^\beta | n}} {1 }  = \sum\limits_{dk \leqslant x} {{d^\beta }\mu (k)\left[ {\frac{y}{{{d^\beta }}}} \right]}  \nonumber \\
   &= y\sum\limits_{dk \leqslant x} {\mu (k)}  - \frac{1}{2}\sum\limits_{dk \leqslant x} {{d^{ \beta }}\mu (k)}  - \sum\limits_{dk \leqslant x} {{d^\beta }\mu (k)\psi \left( {\frac{y}{{{d^\beta }}}} \right)}  \nonumber \\
   &= C_{1,\beta,1}(x,y) + C_{1,\beta,2}(x,y) + C_{1,\beta,3}(x,y), \nonumber
\end{align}
say, and where we have used the definition of $\psi(t)$.
By using \eqref{dbmu} and setting
\begin{equation}\label{exn}
\mathcal{E}(x,n) := 
\begin{cases}
\log x, & \mbox{ if } n=1,\\
1, & \mbox{ if } n>1, \nonumber
\end{cases}
\end{equation}
we can conclude that 
\[C_{1,\beta,2}(x,y) =  - \frac{1}{2}\sum\limits_{dk \leqslant x} {{d^\beta }\mu (k)}  =  - \frac{{{x^{\beta  + 1}}}}{{2(1 + \beta )\zeta (1 + \beta )}} + O({x^\beta }{\mathcal{E}}(x,\beta )).\]
The first sum is independent of $\beta$ since we see that
\begin{align}
  C_{1,\beta,1}(x,y) = y\sum\limits_{m \leqslant x} {\sum\limits_{k|m} {\mu (k)} } = y. \nonumber 
\end{align}
Thus, it remains to compute $C_{1,3}^{(\beta )}(x,y)$ and this will require more effort. We begin by noting that
\[C_{1,\beta,3}(x,y) = \sum\limits_{dk \leqslant x} {{d^\beta }\mu (k)\psi \left( {\frac{y}{{{d^\beta }}}} \right)}  = \sum\limits_{k \leqslant x} {\mu (k)\sum\limits_{d \leqslant x/k} {{d^\beta }\psi \left( {\frac{y}{{{d^\beta }}}} \right)} }. \]
Furthermore, we define the intervals ${I_j}: = ({N_j},2{N_j}]$ where ${N_j} = {N_{j,k}} = \tfrac{x}{k}{2^{ - j}}$ so that $1 \leqslant \tfrac{x}{k}{2^{ - j}}$ implies that $ j \ll \log x$. We may now write
\begin{align}
  C_{1,\beta,3}(x,y) &= \sum\limits_{k \leqslant x} {\mu (k)\sum\limits_{j = 1}^\infty  {\sum\limits_{d \in {I_j}} {{d^\beta }\psi \left( {\frac{y}{{{d^\beta }}}} \right)} } }  = \sum\limits_{k \leqslant x} {\mu (k)\sum\limits_{j \ll \log x} {\sum\limits_{d \in {I_j}} {{d^\beta }\psi \left( {\frac{y}{{{d^\beta }}}} \right)} } }  \nonumber \\
   &\leqslant \sum\limits_{k \leqslant x} {\sum\limits_{j \ll \log x} {\bigg| {\sum\limits_{d \in {I_j}} {{d^\beta }\psi \left( {\frac{y}{{{d^\beta }}}} \right)} } \bigg|} }.  \nonumber
\end{align}
The next step is to apply Abel summation to the inner sum to obtain
\begin{align}
  \sum\limits_{d \in {I_j}} {{d^\beta }\psi \left( {\frac{y}{{{d^\beta }}}} \right)} &= \sum\limits_{d \leqslant 2{N_j}} {{d^\beta }\psi \left( {\frac{y}{{{d^\beta }}}} \right)}  \nonumber \\
   &= {2^\beta }N_j^\beta \sum\limits_{d \leqslant 2{N_j}} {\psi \left( {\frac{y}{{{d^\beta }}}} \right)}  - \int_1^{2{N_j}} {\sum\limits_{d \leqslant t} {\psi \left( {\frac{y}{{{d^\beta }}}} \right)\beta {t^{\beta  - 1}}dt} }  \nonumber \\
   &\ll N_j^\beta \bigg|\sum\limits_{d \leqslant 2{N_j}} {\psi \left( {\frac{y}{{{d^\beta }}}} \right)} \bigg| + \mathop {\sup }\limits_{1 \leqslant t \leqslant 2{N_j}} \bigg|\sum\limits_{d \leqslant t} {\psi \left( {\frac{y}{{{d^\beta }}}} \right)} \bigg|(N_j^\beta  - 1) \nonumber \\
   &\ll N_j^\beta \bigg|\sum\limits_{d \leqslant 2{N_j}} {\psi \left( {\frac{y}{{{d^\beta }}}} \right)} \bigg|. \nonumber 
\end{align}
Therefore we are left with
\[\sum\limits_{d \in {I_j}} {{d^\beta }\psi \left( {\frac{y}{{{d^\beta }}}} \right)}  \ll N_j^\beta \mathop {\sup }\limits_\mathbf{I} \bigg|\sum\limits_{d \in \mathbf{I}} {\psi \left( {\frac{y}{{{d^\beta }}}} \right)} \bigg|,\]
where the supremum is over all subintervals $\mathbf{I}=\{ I_j, \; j=1,\cdots,+\infty \}$. Thus, we have
\[C_{1,\beta,3}(x,y) \ll \sum\limits_{k \leqslant x} {\sum\limits_{j = 1}^\infty  {N_j^\beta \mathop {\sup }\limits_{\mathbf{I}} \bigg| {\sum\limits_{n \in {\mathbf{I}}} {\psi \left( {\frac{y}{{{n^\beta }}}} \right)} } \bigg|} }, \]
where we recall that the sum over $j$ is finite and has $O(\log x)$ terms. Now we use Lemma \ref{exppairlemma}. By taking $k=l=\tfrac{1}{2}$ and seeing that $f(n) = y/n^{\beta} \in \mathbf{F}(N,\infty,\beta+1,y,\varepsilon)$ the exponent pair estimate we need is
\[\sum\limits_{n \in \mathbf{I}} {\psi \left( {\frac{y}{{{n^\beta }}}} \right)}  \ll {y^{\tfrac{1}{3}}}N_j^{\tfrac{{1 - \beta }}{3}} + {y^{ - 1}}N_j^{1 + \beta }.\]
Consequently we have
\begin{align}
  C_{1,\beta,3}(x,y) &\ll \sum\limits_{k \leqslant x} {\sum\limits_{j = 0}^\infty  {N_j^\beta ({y^{\tfrac{1}{3}}}N_j^{\tfrac{{1 - \beta }}{3}} + {y^{ - 1}}N_j^{1 + \beta })} } \nonumber \\ 
	&= \sum\limits_{k \leqslant x} {\sum\limits_{j = 0}^\infty  {({y^{\tfrac{1}{3}}}N_j^{\tfrac{{1 + 2\beta }}{3}} + {y^{ - 1}}N_j^{1 + 2\beta })} }  \nonumber \\
   &\ll \sum\limits_{k \leqslant x} {\bigg( {{y^{\tfrac{1}{3}}}\frac{{{x^{\tfrac{{1 + 2\beta }}{3}}}}}{{{k^{\tfrac{{1 + 2\beta }}{3}}}}} + {y^{ - 1}}\frac{{{x^{1 + 2\beta }}}}{{{k^{1 + 2\beta }}}}} \bigg)} \nonumber \\
	&= {x^{\tfrac{{1 + 2\beta }}{3}}}{y^{\tfrac{1}{3}}}\mathcal{E}\left(x,\tfrac{{1 + 2\beta }}{3}\right) + {x^{1 + 2\beta }}{y^{ - 1}}. \nonumber
\end{align}
This completes the proof.
\section{Acknowledgements}
The first author wishes to acknowledge partial support of SNF grant $200020$-$149150 \backslash 1$.

\end{document}